\documentclass[a4paper]{article}
\usepackage{amsmath,amsthm,amsfonts}
\usepackage{amssymb}
\usepackage[T1]{fontenc}
\usepackage[utf8]{inputenc}

\usepackage{graphicx}
\usepackage{subfigure}
\usepackage{euscript}
\usepackage{units}
\usepackage{enumitem}
\usepackage[bookmarks=true,hidelinks]{hyperref}
\usepackage{cleveref}
\usepackage{bbm}
\usepackage{bm}
\usepackage{xcolor}
\usepackage{calrsfs}

\newtheorem*{theorem}{Theorem}
\newtheorem{lemma}{Lemma}

\newtheorem*{corollary}{Corollary}

\theoremstyle{definition}

\newtheorem*{remark}{Remark}
\newtheorem*{example}{Example}

\def\XXint#1#2#3{{\setbox0=\hbox{$#1{#2#3}{\int}$ }
\vcenter{\hbox{$#2#3$ }}\kern-.56\wd0}}

\renewcommand{\geq}{\geqslant}
\renewcommand{\leq}{\leqslant}

\renewcommand{\epsilon}{\varepsilon}
\newcommand{\eps} {\varepsilon}
\renewcommand{\phi}{\varphi}
\newcommand{\R}{\mathbb{R}}

\DeclareMathOperator{\sgn}{sgn}

\DeclareMathOperator*{\esssup}{ess\, sup}
\DeclareMathOperator*{\essinf}{ess\, inf}
\DeclareMathOperator{\conv}{\overline{conv}}
\newcommand{\ind}{\mathbbm{1}} 
\newcommand{\loc}{{\mathrm{loc}}}

\begin{document}
\title{Sharp uniqueness conditions for one-dimensional, autonomous ordinary differential equations}
\author{U. S. Fjordholm}
\maketitle
\begin{abstract}
We give two conditions that are necessary and sufficient for the uniqueness of Filippov solutions of scalar, autonomous ordinary differential equations with discontinuous velocity fields. When only one of the two conditions is satisfied, we give a natural selection criterion that guarantees uniqueness of the solution.
\end{abstract}

\section{Introduction and statement of the theorem}
The purpose of this note is to derive necessary and sufficient conditions for the uniqueness of Filippov solutions of the scalar, autonomous ordinary differential equation (ODE)
\begin{equation}\label{eq:ode}
\begin{split}
\frac{dX}{dt}(t) &= b(X(t)) \qquad\text{for }t>0  \\
X(0)&=x_0
\end{split}
\end{equation}
where $b:\R\to\R$ is Borel measurable and locally bounded, and $x_0\in\R$. If $b$ is continuous then the sense in which \eqref{eq:ode} holds is classical: $X:[0,\infty)\to\R$ is absolutely continuous and $\frac{d}{dt}X(t)=b(X(t))$ holds for almost every $t>0$. It was shown by Binding \cite{Bin79} that the solution is unique \emph{if and only if} $b$ satisfies the so-called Osgood condition at all zeroes of $b$ (see below). For instance, any Lipschitz continuous $b$ satisfies Osgood's condition. For a general reference on the uniqueness and non-uniqueness of ODEs, see \cite{AgarwalLakshmi}.

If $b$ is merely measurable, say, $b\in L^\infty(\R)$ then the interpretation of \eqref{eq:ode} is more subtle, and choosing a different representative in the equivalence class of $b$ can lead to very different solutions. For instance, redefining the constant velocity field $b(x)\equiv 1$ at a single point, $b(x_0)=0$, yields both the solutions $X(t)\equiv x_0$ and $X(t)=x_0+t$. Several authors have analyzed possible modifications of $b$ on negligible sets in order to ensure existence of a classical solution, see e.g.~\cite{Sho91}. The concept of \emph{Filippov flows} or \emph{Filippov solutions} of \eqref{eq:ode} provides an alternative solution to this issue by choosing a canonical representation of the velocity field. More precisely, the differential equation \eqref{eq:ode} is replaced by a differential \emph{inclusion} where the right-hand side contains information on the behavior of $b$ in an infinitesimal neighborhood of $X(t)$. Filippov \cite{Fil60} showed that there exists a Filippov solution of \eqref{eq:ode} under very mild conditions on $b$, for instance if $b\in L^\infty(\R)$ or, for local existence, $b\in L^\infty_\loc(\R)$.
%  has at most linear growth,
%\begin{equation}\label{eq:lingrowth}
%|b(x)| \leq C(1+|x|) \qquad \forall\ x\in\R.
%\end{equation}
%This condition can be weakened slightly; see REMARK ON THIS LATER IN THE PAPER!

The main theorem of this paper, stated in Section \ref{sec:mainthm}, gives necessary and sufficient conditions for the uniqueness of Filippov solutions of \eqref{eq:ode}. As a corollary we define a class of functions $\tilde{b}:\R\to\R$ for which the corresponding ODE all have the same unique, \emph{classical} solution.

In Section \ref{sec:filippovsoln} we provide the definition of Filippov solutions and in Section \ref{sec:osgood} we describe the essential Osgood criterion. Section \ref{sec:proof} contains the proof of the Theorem and its Corollary, while Section \ref{sec:examples} lists some examples.

%If $b$ is continuous and bounded then the classical result by Peano \cite{Pea90} gives existence of at least one classical solution of \eqref{eq:ode}, that is, a function $X:[0,\infty)\to\R$, $X(0)=x_0$, satisfying the differential equation \eqref{eq:ode} for almost every $t>0$. It was shown much later by Binding \cite{Bin79} that the solution is unique \emph{if and only if} $b$ satisfies the so-called Osgood condition at all zeroes of $b$ (see below). For instance, any Lipschitz continuous $b$ satisfies Osgood's condition.
%More precisely, for every point $x_0$ where $b(x_0)=0$, there must be some $\delta_0>0$ such that
%\begin{equation}\label{eq:osgood}
%\int_I \frac{1}{\big(\sgn(z)b(x_0+z)\big)^+}\,dz = +\infty \qquad \forall\ \text{intervals } 0\in I\subset(-\delta_0,\delta_0).
%\end{equation}
%For instance, the fact that $z \mapsto \frac{1}{|z|}$ is non-integrable at the origin implies that Lipschitz continuous $b$ satisfy \eqref{eq:osgood}
%As the concept of classical solutions is sensitive to modifications of $b$ on sets of measure 0, the existence of classical solutions might fail if $b$ is allowed to be discontinuous, say, if $b$ merely lies in some Lebesgue space. The concept of Filippov solutions resolves this issue by replacing the differential equation \eqref{eq:ode} by a differential inclusion where the right-hand side $b(X(t))$ is replaced by a set-valued function which is independent of changes to $b$ on negligible sets \cite{Fil60}. As shown by Filippov, a solution in this sense exists for large classes of velocity fields, such as $b\in L^\infty(\R)$.

\subsection{Set-valued functions and Filippov solutions}\label{sec:filippovsoln}
We say that an absolutely continuous function $X:[0,T)\to\R$ is a \emph{Filippov solution} of \eqref{eq:ode} if $X(0)=x_0$ and
\[
\frac{dX}{dt}(t) \in K[b](X(t)) \qquad \text{for a.e. } t\in(0,T)
\]
(see \cite{Fil60}). Here, the set-valued function $K[b]$ is defined as
\[
K[b](x) := \bigcap_{\delta>0}\bigcap_{\substack{N\subset\R \\ |N|=0}} \conv\big(b\big(B_\delta(x)\setminus N\big)\big)
\]
where $B_\delta(x)$ is the open ball around $x$ with radius $\delta$, and $\conv(A)$ is the smallest closed, convex set containing $A$. 
%It can be shown that for every $x$ there exists a set $N_0\subset\R^d$ with zero measure such that
%\begin{equation}\label{eq:altdefK}
%K[b](x) = \bigcap_{\delta>0} \conv\big(b\big(B_\delta(x)\setminus N_0\big)\big),
%\end{equation}
%see \cite{Fil60}. (Indeed, $N_0$ can be chosen as the complement of the set of points of approximate continuity of $f$.) 
In a similar vein we define the essential upper and lower bounds of $b$ at $x$ as
\begin{equation}\label{eq:defessminmax}
\begin{split}
m[b](x) := \min\big(K[b](x)\big) = \lim_{\delta\to0}\essinf_{x'\in B_\delta(x)} b(x'), \\
M[b](x) := \max\big(K[b](x)\big) = \lim_{\delta\to0}\esssup_{x'\in B_\delta(x)} b(x').
\end{split}
\end{equation}
We will say that $b$ is continuous at a point $x$ if the set $K[b](x)$ contains a single point, otherwise we say that $b$ is discontinuous at $x$. It is evident that this coincides with the usual definition of continuity at a point, possibly after redefining $b$ on a negligible set. %At points of continuity $x$ we will define $b(x)$ as the value contained in the singleton $K[b](x)$; it can be shown that this only modifies $b$ on a set of measure zero (SHOW THIS!).

%\begin{lemma}
%Let $b\in L^\infty(\R^d;\R^d)$. Then $K[b]$ is upper hemicontinuous, i.e.~for every $x\in\R^d$ and every neighborhood $V\supset K[b](x)$, there is a neighborhood $U$ of $x$ such that $K[b](U)\subset V$.
%\end{lemma}
%\begin{proof}
%Let $N_0$ be as in \eqref{eq:altdefK} and let $\delta>0$ be such that $\conv\big(b\big(B_\delta(x)\setminus N_0\big)\big) \subset V$. If $y\in B_\delta(x)$ then also $\conv\big(b\big(B_{\delta'}(y)\setminus N_0\big)\big) \subset V$, where $\delta' = \delta-|x-y|>0$. It follows that $K[b](y)\subset V$.
%\end{proof}

We list below some properties that are straightforward to check (see also \cite{AubFra,Dei}):
\begin{enumerate}[label=(\roman*)]
\item $K[b]$ is upper hemicontinuous.
\item If $0\notin K[b](x)$ for some $x\in\R$ then there is a neighborhood $U$ of $x$ such that $0 \notin K[b](y)$ for every $y\in U$.
\item $m[b]$ and $M[b]$ are lower and upper semi-continuous, respectively.
\item The set of discontinuities of $b$ coincides with the measurable set $\{x\ :\ m[b](x)<M[b](x)\}$.
\end{enumerate}

\subsection{The Osgood condition}\label{sec:osgood}
The classical uniqueness result for ODEs requires Lipschitz continuity of the velocity field $b$. In 1898 Osgood relaxed this condition to mere continuity of $b$ along with an integrability condition on its reciprocal \cite{Osg98}. We recall the main idea of Osgood's condition here. We will call a function $g:(-\delta_0,\delta_0)\to[0,\infty)$ an \emph{Osgood function} if it is nonnegative, Borel measurable and satisfies 
\begin{equation}\label{eq:osgoodfunc}
\int_{-\delta}^0 g(z)^{-1}dz = +\infty, \qquad \int_0^\delta g(z)^{-1}dz = +\infty \qquad \forall\ \delta\in(0,\delta_0).
\end{equation}

\begin{lemma}[Osgood lemma]\label{lem:osgood}
Let $g:(-\delta_0,\delta_0)\to[0,\infty)$ be an Osgood function and let $u:[0,T)\to(-\delta_0,\delta_0)$ be an absolutely continuous function satisfying $u(0)=0$ and
\begin{equation}\label{eq:osgoodcondition}
\frac{d}{dt}|u(t)| \leq g(u(t)) \qquad \text{for a.e. } t\in(0,T).
\end{equation}
Then $u \equiv 0$.
\end{lemma}
\begin{proof}
Assume conversely that, say, $u(t_1)>0$ for some $t_1>0$, and let $t_0\in[0,t_1)$ be such that $u(t_0)=0$ but $u(t)>0$ for all $t\in(t_0,t_1]$. From \eqref{eq:osgoodcondition} it follows in particular that $\frac{du}{dt}(t) \leq g(u(t))+\eps$ for every $\eps>0$ and a.e.\ $t\in(t_0,t_1)$. Dividing this inequality by the right-hand side and integrating over $t\in(t_0,t_1)$ yields
\[
t_1-t_0 \geq \int_{t_0}^{t_1} \frac{\frac{du}{dt}(t)}{g(u(t))+\eps}dt = \int_0^{u(t_1)}\frac{1}{g(z)+\eps}dz.
\]
But the right-hand side goes to $+\infty$ as $\eps\to0$, a contradiction.
\end{proof}

\subsection{The main theorem}\label{sec:mainthm}
In Section \ref{sec:proof} we prove the following uniqueness result. We recall from \cite{Fil60} that if $b\in L^\infty_\loc(\R)$ then there exists at least one local-in-time Filippov solution of \eqref{eq:ode}.
\begin{theorem}\label{thm:uniqueode1d}
Assume that $b\in L^\infty_\loc(\R)$ satisfies the following two conditions:
\begin{enumerate}[label=(\Alph*)]
\item\label{item:discontcond} the set 
\begin{equation}\label{eq:discontcondition}
\big\{x\in\R\ :\ 0 \notin K[b](x) \text{ and $b$ is discontinuous at $x$} \big\}
\end{equation}
has zero Lebesgue measure
\item\label{item:osl} for every $x\in\R$ where $0\in K[b](x)$, the function
\begin{equation}\label{eq:osl}
g(z):=M\big[b_x^+\big](z), \qquad \text{where } b_x^+(z) := \big(b(x+z)\sgn(z)\big)^+
\end{equation}
is an Osgood function.  (Here, $u^+ = \max(0,u)$.)
\end{enumerate}
Then the Filippov solution of \eqref{eq:ode} is unique. Conversely, if one of the conditions \ref{item:discontcond} or \ref{item:osl} does not hold then there is some $x_0\in\R$ for which there are uncountably many Filippov solutions.
\end{theorem}

If $b$ is continuous then the Theorem reduces to \cite[Theorem 5.3]{Bin79}. Indeed, if $b$ is continuous then the set \eqref{eq:discontcondition} is empty, and condition \ref{item:osl} is equivalent to saying that $b_x^+$ is an Osgood function whenever $b(x)=0$.

Condition \ref{item:discontcond} addresses a deficiency in the concept of Filippov solutions: Roughly speaking, if discontinuities in $b$ are too densely packed then the set-valued function $K[b]$ is unable to ``see'' the (almost everywhere defined) function $b$. An example where \ref{item:discontcond} is violated is given in Section \ref{sec:examples}. The following corollary shows that when \ref{item:discontcond} is violated, the additional requirement of being a classical solution can act as a selection criterion among the infinitely many solutions.

\begin{corollary}
Assume that $b\in L^\infty_\loc(\R)$ satisfies condition \ref{item:osl}. Define
\begin{equation*}
\begin{split}
\mathcal{L}_b := \Big\{\tilde{b}:\R\to\R\ :\ \ &b=\tilde{b} \text{ a.e.,}\quad \tilde{b}(x)\in K[b](x)\ \forall\ x\in\R,\\
&0\in K[b](x)\ \Rightarrow\ \tilde{b}(x)=0 \Big\}.
\end{split}
\end{equation*}
%\begin{itemize}
%\item[(i)] 
Then for every $x_0\in\R$ and $\tilde{b}\in\mathcal{L}_b$ there is a unique \emph{classical} solution of the ODE
\begin{equation}\label{eq:ode2}
\begin{split}
\frac{dX}{dt}(t) &= \tilde{b}(X(t)) \qquad\text{for }t>0  \\
X(0)&=x_0
\end{split}
\end{equation}
and this solution is independent of the choice of $\tilde{b}\in\mathcal{L}_b$.
%\item[(ii)] Let $X_t^{x_0}$ denote the classical solution of \eqref{eq:ode2}. If $b$ also satisfies
%\begin{equation}\label{eq:lineargrowth}
%|b(x)| \leq C(1+|x|) \qquad \forall\ x\in\R
%\end{equation}
%then $x\mapsto X_t^x$ is a homeomorphism on $\R$ for every $t\geq0$. \textbf{TODO: PROVE THIS}
%\end{itemize}
\end{corollary}
As mentioned above, a classical solution is an absolutely continuous function $X:[0,T)\to\R$ satisfying the equation \eqref{eq:ode2} for a.e.~$t\in(0,T)$. The proof of the Corollary is given in Section \ref{sec:proof}.
\begin{remark}
Let $X_t^{x_0}$ denote the (classical or Filippov) solution of \eqref{eq:ode}. Using the fact that \eqref{eq:ode} is time-reversible one can show that if the solution $X_t^{x_0}$ is unique (as in the Theorem and in the Corollary), then the map $x\mapsto X_t^x$ is continuous and surjective for any $t\geq0$. 
\end{remark}

\section{Proof of the main theorem}\label{sec:proof}
\begin{proof}[Proof of sufficiency of \ref{item:discontcond}, \ref{item:osl}]
It is sufficient to show that the solution remains unique up to some time $T>0$. For any given $x_0\in\R$ there are two cases to consider. 

\noindent
\textbf{Case 1:} First, assume that $0 \in K[b](x_0)$ and let $X=X(t)$ be a solution of \eqref{eq:ode}. We wish to show that $X(t)\equiv x_0$, so assume conversely that $X(t)\neq x_0$ for some $t>0$; by translating in time we may assume that $X(t)>x_0$ for all $t$ in some interval $(0,T)$. (The case $X(t)<x_0$ is completely symmetric.) By definition, the function $g$ in \eqref{eq:osl} is nonnegative and upper semi-continuous. By the definition of $K[b](x)$ we have for every $\delta>0$
\[
\frac{dX}{dt}(t) \leq \esssup_{x'\in B_{\delta}(X(t))}b(x'),
\]
and for every $\eps>0$ there is a subset of points $x^*\in B_{\delta}(x(t))$ of positive measure such that
\[
\esssup_{x'\in B_{\delta}(X(t))}b(x') \leq b(x^*) + \eps.
\]
It follows that for almost every $t\in(0,T)$ and for sufficiently small $\delta$,
\begin{align*}
\frac{d}{dt}|X(t)-x_0| &= \frac{dX}{dt}(t) \leq b(x^*) + \eps \\
&= \sgn(x^*-x_0)b(x^*) + \eps  \\
&\leq g(x^*-x_0) + \eps &&\text{{(by \eqref{eq:osl})}.}
\end{align*}
Passing $\delta,\eps\to0$ and using the upper semi-continuity of $g$ we arrive at $\frac{d}{dt}|X(t)-x_0| \leq g(X(t)-x_0)$. Applying \Cref{lem:osgood} yields $X(t)= x_0$ for every $t\in[0,T)$. Thus, the constant solution is unique.

\noindent
\textbf{Case 2:} 
Assume now that $0\notin K[b](x_0)$; since $K[b](x_0)$ is closed and convex we may assume that, say, $K[b](x_0)\subset(0,\infty)$. Then there is a $\delta>0$ such that $K[b](x)\subset[c,\infty)$ for some $c>0$ for every $x\in B_{\delta}(x_0)$. Write $d=\|b\|_{L^\infty(B_{\delta}(x_0))}$. Let now $X,Y$ be two solutions of \eqref{eq:osl}, and fix $T>0$ such that $X(t),Y(t)\in B_{\delta}(x_0)$ for all $t\in[0,T)$. Let $A_X,A_Y\subset [0,T)$ the the sets of differentiability of $X,Y$, respectively, both of which have full measure.

Since $\frac{dX}{dt}(t),\frac{dY}{dt}(t)\in[c,d]$ we have $X(t),Y(t)\in[x_0+ct,x_0+dt]$ for all $t\in[0,T)$ and hence---possibly after decreasing $T$---there is a map $\tau:[0,T)\to[0,T)$ such that $\tau(0)=0$ and $X(\tau(t))=Y(t)$ for all $t\in[0,T)$. Since $X,Y$ are absolutely continuous, so is $\tau$, and moreover $\tau'(t) \geq \frac{d}{c} > 0$. It follows that the set $A = A_Y \cap \tau^{-1}(A_X)\subset[0,T)$ has full measure. Finally, define
\[
E = A \cap \big\{t\in [0,T)\ :\ K[b](Y(t)) \text{ is a singleton} \big\}.
\]
By assumption 1 and the fact that $Y$ is monotone, the set $E$ also has full measure. For \emph{every} $t\in E$ we have now
\[
\frac{dX}{dt}(\tau(t)) = b(X(\tau(t))) = b(Y(t)) = \frac{dY}{dt}(t).
\]
But at the same time $X(\tau(t))=Y(t)$, so that $\frac{dX}{dt}(\tau(t))\tau'(t)=\frac{dY}{dt}(t)$ for a.e.~$t$. It follows that $\tau'(t)\equiv 1$, whence $X(t)=Y(t)$ for all $t$.
\end{proof}

Next, we claim that condition \ref{item:discontcond} is necessary. To this end we need the following elementary result.

\begin{lemma}\label{lem:positivedensity}
Let $U\subset\R$ be an open set and let $K\subset U$ be a measurable set with $|K|>0$. Then there exists a point $x_0\in U$ such that $|[x_0,x_0+\delta)\cap K| > 0$ for every $\delta>0$.
\end{lemma}
\begin{proof}
Select an interval $[a,b)\subset U$ such that $|[a,b)\cap K| > 0$. Define
\[
x_0 = \sup\big\{x\in[a,b)\ :\ |[a,x)\cap K| = 0\big\}.
\]
Then $a\leq x_0<b$ so $x_0\in U$, and $|[x_0,x_0+\delta)\cap K| > 0$ for every $\delta>0$.
\end{proof}

\begin{proof}[Proof of necessity of \ref{item:discontcond}]
Assume that \ref{item:discontcond} is \emph{not} satisfied, and define
\begin{align*}
&D :=\big\{x\in\R\ :\ b \text{ is discontinuous at $x$} \big\}, \\
U^- := \big\{x\in\R\ :\ &K[b](x)\subset(-\infty,0)\big\}, \qquad
U^+ := \big\{x\in\R\ :\ K[b](x)\subset(0,\infty)\big\}.
\end{align*}
By assumption, at least one of the sets $D^-:=U^-\cap D$ and $D^+:=U^+\cap D$ has positive measure, so we assume that, say, $|D^+|>0$. Let $x_0\in U^+$ be a point where $\big|[x_0,x_0+\delta)\cap D^+\big| > 0$ for every $\delta>0$ (cf.\ \Cref{lem:positivedensity}). Since $K[b](x_0)\subset(0,\infty)$ there is a $c>0$ and a $\delta_0>0$ such that $K[b](x)\subset[c,\infty)$ for every $|x-x_0|<\delta_0$.
In particular, $c \leq m[b](x) < M[b](x)$ for every $x\in [x_0,x_0+\delta_0)\cap D^+$. Select measurable functions $b_1,b_2$ such that $m[b](x)\leq b_1(x) \leq b_2(x) \leq M[b](x)$ and such that $b_1(x)<b_2(x)$ for every $x\in [x_0,x_0+\delta_0)\cap D^+$. Note that there are uncountably many such pairs of functions. Then the functions $X_1,X_2$ defined by
\[
X_i(t) = G_i^{-1}(t), \qquad G_i(x) := \int_{x_0}^x \frac{1}{b_i(y)}\,dy
\]
are distinct Filippov solutions of \eqref{eq:ode}.
\end{proof}

\begin{proof}[Proof of necessity of \ref{item:osl}]
Let $g$ be as in \eqref{eq:osl}. If condition \ref{item:osl} is not satisfied then necessarily either $\int_{-\delta}^0 g(z)^{-1}\,dz < \infty$ or $\int_0^\delta g(z)^{-1}\,dz < \infty$ for every $0<\delta\leq\delta_0$ for a sufficiently small $\delta_0>0$. Assume the latter case; the former is completely analogous. Necessarily, $g(\delta)>0$ for almost every $\delta\in(0,\delta_0)$, so in particular $g(z) = M\big[b_{x_0}^+\big](z) = M[b](x_0+z)$. Let
\[
G(x) := \int_{0}^{x-x_0}\frac{1}{g(z)}dz, \qquad x\in[x_0,x_0+\delta_0).
\]
Then $G$ is absolutely continuous with $G'(x)=\frac{1}{g(x-x_0)} \geq \frac{1}{\|b\|_{L^\infty}} > 0$, so $G$ is invertible and its inverse $X:=G^{-1}:[0,T)\to [x_0,x_0+\delta_0)$ (for some $T>0$) is also absolutely continuous. Differentiating the relation $\int_0^{X(t)-x_0}\frac{1}{g(z)}dz = t$ and reorganizing, we find that
\[
\frac{dX}{dt}(t) = g\big(X(t)-x_0\big) = M[b]\big(X(t)\big) \in K[b]\big(X(t)\big),
\]
whence $X$ solves \eqref{eq:ode}. Clearly, the trivial solution $X_0(t)\equiv x_0$ is also a solution of \eqref{eq:ode}; hence, any $c>0$ yields a new solution
\[
X_c(t) = \begin{cases}
x_0 & \text{for } 0\leq t\leq c \\
X(t-c) & \text{for } c\leq t < c+T.
\end{cases}
\]
We conclude that there exists a continuum of solutions of \eqref{eq:ode}.
\end{proof}

We conclude this section with the proof of the Corollary.
\begin{proof}[Proof of the Corollary]
From the definition of $\mathcal{L}_b$ it is clear that any classical solution of \eqref{eq:ode2} is also a Filippov solution of \eqref{eq:ode}. Hence, the fact that $b$ satisfies condition \ref{item:osl} implies that if $0\in K[b](x_0)$ then any classical solution must satisfy $X(t) \equiv x_0$. If $0\notin K[b](x_0)$, say, if $K[b](x)\subset[c,\infty)$ for $|x-x_0|<\delta$ for some $c>0$, then define $X(t)=G^{-1}(t)$, where
\[
G(x):=\int_{x_0}^x\frac{1}{b(z)}\,dz = \int_{x_0}^x\frac{1}{\tilde{b}(z)}\,dz.
\]
Then $X$ satisfies \eqref{eq:ode2} in the classical sense for a.e.\ $t$, and is necessarily the only classical solution since any other $Y(t)$ such that $\frac{d}{dt}Y(t)=\tilde{b}(Y(t))$ for a.e.\ $t$ also satisfies $G(Y(t))=t$, whence $X=Y$.
\end{proof}

\section{Examples}\label{sec:examples}

\begin{example}[Velocity fields not satisfying \ref{item:osl}]
Counterexamples to uniqueness of \eqref{eq:ode} when $b$ does not satisfy the Osgood condition are well-known, the most popular ones being $b(x) = |x|^\alpha$ for some $\alpha\in(0,1)$ and the Heaviside function $b(x)=\ind_{(0,\infty)}(x)$. Note that, say, $b(x)=1+|x|^\alpha$, does satisfy conditions \ref{item:discontcond} and \ref{item:osl}, even though it is not (one-sided) Lipschitz bounded. The function $b(x)=-x\log |x|$ is an example of a non-Lipschitz function which does satisfy condition \ref{item:osl}.
\end{example}

\begin{example}[An everywhere discontinuous velocity field]
Let $A\subset\R$ be a Borel set with the following property: For every $x\in\R$ and $\delta>0$, both $|A\cap B_{\delta}(x)|>0$ and $|B_{\delta}(x)\setminus A| > 0$ (see Rudin \cite{Rud83}). Define 
\[
b(x) = \begin{cases} 1 & x\in A\\ 2 & x\notin A.\end{cases}
\]
It is easy to check that $K[b](x) \equiv [1,2]$, and hence $b$ is nowhere continuous. Clearly, $x(t) = x_0+at$ is a Filippov solution of \eqref{eq:ode} for any $a\in[1,2]$. Note that, by the Corollary, there is a unique classical solution for the above velocity field.
\end{example}

\end{document}